\documentclass[a4paper,twoside,12pt]{article}
\usepackage{amssymb,amsfonts,amsmath,amsthm,latexsym}
\usepackage[pdftex,bookmarks,colorlinks=false]{hyperref}
\usepackage[hmargin=1.2in,vmargin=1.2in]{geometry}
\usepackage[auth-sc-lg,affil-sl]{authblk}
\usepackage{graphicx,caption,subcaption}
\newtheorem{theorem}{Theorem}[section]
\newtheorem{corollary}[theorem]{Corollary}
\newtheorem{definition}[theorem]{Definition}
\newtheorem{lemma}[theorem]{Lemma}
\newtheorem{problem}[theorem]{Problem}
\newtheorem{proposition}[theorem]{Proposition}
\newtheorem{remark}[theorem]{Remark}

\numberwithin{equation}{section}

\def\ni{\noindent} 

\title{\textbf{\sc Certain Types of Total Irregularities of Graphs and Digraphs}}

\author{Johan Kok}
\affil{\small Tshwane Metropolitan Police Department\\ City of Tshwane, Republic of South Africa\\ E-mail: kokkiek2@tshwane.gov.za}

\author{Sudev Naduvath}
\affil{\small Department of Mathematics\\ Vidya Academy of Science \& Technology \\ Thalakkottukara, Thrissur - 680501, India.\\ E-mail: sudevnk@gmail.com}

\date{}

\begin{document}
\maketitle

\begin{abstract}
The total irregularity of a simple undirected graph $G$ is denoted by $irr_t(G)$ and is defined as $irr_t(G) =\frac{1}{2}\sum \limits_{u,v \in V(G)}|d(u) - d(v)|$. In this paper, the concept called edge-transformation in relation to total irregularity of simple undirected graphs with at least one cut edge is introduced. We also introduce the concept of an edge-joint between two simple undirected graphs. We also introduce the concept of total irregularity in respect of in-degree and out-degree  in simple directed graphs. These invariants are called total in-irregularity and total out-irregularity  respectively.  In this paper, we initiate a study on these parameters of given simple undirected graphs and simple digraphs.
\end{abstract}

\ni \textbf{Key Words:} Total irregularity, branch-transformation, edge-transformation, edge-joint, total in-irregularity, total out-irregularity.

\vspace{0.2cm}

\ni \textbf{Mathematics Subject Classification: 05C07, 05C20, 05C38, 05C70, 05C75.}

\section{Introduction}

For general notations and concepts in graph theory, we refer to \cite{BM1}, \cite{FH} and \cite{DBW} and for digraph theory, we further refer to \cite{CL1} and \cite{JG1}. All graphs mentioned in this paper are simple, connected and finite graphs, unless mentioned otherwise. Also, except for Section \ref{Sec-4}, all the graphs mentioned here are undirected graphs.

A graph $G$ is said to be \textit{regular} if the degree of all vertices are equal. A graph that is not regular is called an \textit{irregular graph}. The total irregularity of a given simple connected graph is defined in \cite{MOA} as follows.

\begin{definition}{\rm 
\cite{MOA} The \textit{imbalance} of an edge $e=uv$ in a given graph $G$ is defined as $|d(u)-d(v)|$. The \textit{total irregularity} of a graph $G$,  denoted by $irr_t(G)$, is defined as $irr_t(G) =\frac{1}{2}\sum \limits_{u,v \in V(G)}|d(u)-d(v)|$. }
\end{definition}

If the vertices of a graph $G$ on $n$ vertices are labelled as $v_i, i = 1, 2, 3, \ldots, n$, then the definition may be $irr_t(G) = \frac{1}{2}\sum \limits_{i=1}^{n} \sum \limits_{j=1}^{n}|d(v_i) - d(v_j)| = \sum \limits_{i=1}^{n} \sum \limits_{j=i+1}^{n}|d(v_i) - d(v_j)|$ or $\sum \limits_{i=1}^{n-1} \sum \limits_{j=i+1}^{n}|d(v_i) - d(v_j)|$. For a graph on a singular vertex (\emph{1-null} graph or $K_1$), we define $irr_t(G) = 0$. 
 Clearly, $irr_t(G) = 0$ if and only if $G$ is regular.

The notion of \textit{branch-transformation} of a graph has been introduced in \cite{ZYY} as follows.

\begin{definition}{\rm 
\cite{ZYY} Let $G$ be a graph with at least two pendent vertices. Without loss of generality, let $u$ be a vertex of $G$ with $d_G(u)\ge 3$, $T$ be a hanging tree of $G$ connecting to $u$ with $|V(T)|\ge 1$ and $v$ be a pendant vertex of $G$ with $v\notin T$. Let $G'$ be the graph obtained from $G$ by deleting $T$ from vertex $u$ and attaching it to vertex $v$. We call the transformation from $G$ to $G'$ a \textit{branch-transformation} on $G$ from vertex $u$ to vertex $v$.}
\end{definition}

Certain studies on irregularities and total irregularities of given graphs and the properties graphs related to these irregularities have been studied in \cite{ABD,ACD,AD,MOA,DS1,HR} and \cite{ZYY}. Motivated from these studies, in this paper, we initiate a study on certain types of total irregularities of given graphs and directed graphs. We introduce the notion of edge-transformation in relation to total irregularity of simple graphs with at least one cut edge as well as an edge-joint between two graphs. We also introduce the notion of total irregularity with respect to in-degree and out-degree in directed graphs.

\section{Total Irregularity Resulting from Edge-Joints}

Consider a graph $G$ on $n$ vertices with two connected components $G_1$ and $G_2$. Therefore, $G = G_1 \cup G_2$.  Hence, the total irregularity of $G$ is given by $irr_t(G) = irr_t(G_1) + irr_t(G_2) + \sum\limits_{i=1}^{r}\sum\limits_{j=1}^{s}|d(u_i)- d(v_j)|$, where $u_i \in V(G_1)$, $v_j \in V(G_2)$ and $r=|V(G_1)|$ and $s=|V(G_2)|$.

The concept of an edge-joint between two simple undirected graphs $G$ and $H$ is defined below.

\begin{definition}{\em 
The \textit{edge-joint} of two graphs $G$ and $H$ is the graph obtained by adding one edge, say $uv$, where $u \in V(G), v \in V(H)$, and is denoted by $G\rightsquigarrow_{uv}H$.}
\end{definition}

\begin{remark}
It is to be noted that $G\rightsquigarrow_{uv}H = G \cup H + uv$ and $G\rightsquigarrow_{uv}H \simeq H\rightsquigarrow_{vu}G$.
\end{remark}

\ni Now, we make the following standard definitions and notation.

\begin{definition}\label{D-VSS1}{\rm 
Let $G$ be a graph on $n$ vertices with two connected components $G_1$ and $G_2$ whose vertex sets are $V(G_1) = \{u_i: 1 \leq i \leq r\}$ and $V(G_2) = \{v_j: 1 \leq j \leq s\}$. We fix the vertices $u_1$ from $G_1$ and $v_1$ from $G_2$. Now, we define the vertex subsets $V_1 = \{u_x: d_{G_1}(u_x) \le d_{G_1}(u_1), x\ne 1\}$; $V_2 = \{u_y: d_{G_1}(u_y) > d_{G_1}(u_1)\}$ and let $|V_1|=a$ and  $|V_2|=b$. Then, choose $V_3 = \{v_x:d_{G_2}(v_x) \le d_{G_1}(u_1)\}$ and $V_4 = \{v_y: d_{G_2}(v_y) > d_{G_1}(u_1)\}$, where $|V_3| = a^{\ast}$ and $|V_4| = b^{\ast}$. Similarly, let $V_5 = \{v_z: d_{G_2}(v_z) \le d_{G_2}(v_1),~z \ne 1\}$ and $V_6 = \{v_w: d_{G_2}(v_w) > d_{G_2}(v_1)\}$ where $|V_5| = c$ and $|V_6| = d$ and choose $V_7= \{u_z: d_{G_1}(u_z) \le d_{G_2}(v_1)\}$ and $V_8 = \{u_w: d_{G_1}(u_w) > d_{G_2}(v_1)\}$ where $|V_7| = c^{\ast}$ and $|V_8| = d^{\ast}$.}
\end{definition}

In view of the above definition, we have the relation between the cardinalities of the the above mentioned vertex subsets of $V(G_1)$ and $V(G_2)$ as follows.
 
\begin{remark}\label{Rem-2.3}{\rm 
With regard to definition \ref{D-VSS1}, define the variables $b=r-a$, $d=s-c=n-r-c$, $b^{\ast} = r-a^{\ast}$ and $d^{\ast} = s-c^{\ast} = n-r-c^{\ast}$.}
\end{remark}

\begin{theorem}\label{Thm-2.1}
Let $G$ be a graph on $n$ vertices with two connected components $G_1$ and $G_2$, where $V(G_1) =\{u_i: 1\le i \le r\}$ and $V(G_2) =\{v_j: 1\le j\le s\}$ . Also, let $G' = G_1\rightsquigarrow_{u_1v_1}G_2$. Then, we have  
$irr_t(G') = irr_t(G_1) + irr_t(G_2) + \sum \limits_{i=1}^{r} \sum\limits_{j=1}^{s}|d_{G_1}(u_i)- d_{G_2}(v_j)| + 2n-2(b + b^{\ast}+d+d^{\ast})-2$  or
$irr_t(G') = irr_t(G_1) + irr_t(G_2) + \sum \limits_{i=1}^{r} \sum\limits_{j=1}^{s}|d_{G_1}(u_i)- d_{G_2}(v_j)| + 2(a + a^{\ast} + c + c^{\ast}) -2n + 2$. 
\end{theorem}
\begin{proof}
Clearly, for the graph $G = G_1 \cup G_2$, we have $irr_t(G) = irr_t(G_1) + irr_t(G_2) + \sum \limits_{i=1}^{r} \sum\limits_{j=1}^{s}|d_{G_1}(u_i)- d_{G_2}(v_j)|$ with $|V(G_1)| = r$ and $|V(G_2)| = s$. 

By increasing $d_{G_1}(u_1)$ by 1 we increase the partial sum $\sum \limits_{\substack{j=1\\w_j \in V_1}}^{a}|d_{G_1}(u_1) - d_{G_1}(w_j)|$ by exactly $(a - 1)$. It also reduces the partial sum $\sum \limits_{\substack{j=1\\w_j \in V_2}}^{b}|d_{G_1}(u_1) - d_{G_1}(w_j)|$ by exactly $b$. It also increases the partial sum $\sum \limits_{\substack{j=1\\w_j \in V_3}}^{a^{\ast}}|d_{G_1}(u_1) - d_{G_2}(w_j)|$ by exactly $a^{\ast}$ and decreases the partial sum $\sum \limits_{\substack{j=1\\w_j \in V_4}}^{b^{\ast}}|d_{G_1}(u_1) - d_{G_2}(w_j)|$ by exactly $b^{\ast}$. Furthermore, by increasing $d_{G_2}(v_1)$ by 1, we increase the partial sum $\sum \limits_{\substack{j=1\\w_j \in V_5}}^{c}|d_{G_2}(v_1) - d_{G_2}(w_j)|$ by exactly $(c - 1)$. It also reduces the partial sum $\sum \limits_{\substack{j=1\\w_j \in V_6}}^{d}|d_{G_1}(u_1) - d_{G_2}(w_j)|$ by exactly $d$. It also increases the partial sum $\sum \limits_{\substack{j=1\\w_j \in V_7}}^{c^{\ast}}|d_{G_1}(u_1) - d_{G_1}(w_j)|$ by exactly $c^{\ast}$ and decreases the partial sum $\sum \limits_{\substack{j=1\\w_j \in V_8}}^{d^{\ast}}|d_{G_1}(u_1) - d_{G_1}(w_j)|$ by exactly $d^{\ast}$. 

Hence, we have an interim result as follows.

$irr_t(G') = irr_t(G_1) + irr_t(G_2) + \sum \limits_{i=1}^{r} \sum\limits_k{j=1}^{s}|d_{G_1}(u_i)- d_{G_2}(v_j)| + (a - 1) - b + a^{\ast} - b^{\ast} + (c - 1)-d + c^{\ast} - d^{\ast}\\  =  irr_t(G_1) + irr_t(G_2) + \sum \limits_{i=1}^{r} \sum\limits_{j=1}^{s}|d_{G_1}(u_i)- d_{G_2}(v_j)| + (a-b) + (a^{\ast}- b^{\ast}) + (c - d) + (c^{\ast} - d^{\ast}) - 2$.

By substituting the variables $b, d, b^{\ast}$, and $d^{\ast}$ as defined in Definition \ref{D-VSS1} the final result is as follows.

$irr_t(G') = irr_t(G_1) + irr_t(G_2) + \sum \limits_{i=1}^{r} \sum\limits_{j=1}^{s}|d_{G_1}(u_i)- d_{G_2}(v_j)| + 2n - 2(b + b^{\ast} + d + d^{\ast}) - 2$, or; 
$irr_t(G') = irr_t(G_1) + irr_t(G_2) + \sum \limits_{i=1}^{r} \sum\limits_{j=1}^{s}|d_{G_1}(u_i)- d_{G_2}(v_j)| + 2(a + a^{\ast} + c + c^{\ast}) -2n + 2$, follows. 
\end{proof}

\ni Clearly $irr_t(G')$ is edge dependent in general but we have the following Corollary.

\begin{corollary}\label{Cor-2.6}
Let the degree sequence of graphs $G_1$ and $G_2$ be $(d_{G_1}(u_1) \leq d_{G_1}(u_2) \leq d_{G_1}(u_3) \leq \ldots\leq d_{G_1}(u_n))$ and $(d_{G_2}(v_1) \leq d_{G_2}(v_2) \leq d_{G_2}(v_3) \leq \ldots\ldots \leq d_{G_2}(v_m))$ respectively. If $d_{G_1}(u_i) = d_{G_2}(v_j)$ for some $i,j$ and $d_{G_1}(u_k) = d_{G_2}(v_l)$ for some $k,l$ and $G' = G_1\rightsquigarrow_{u_iv_l}G_2$ and $G'' = G_1\rightsquigarrow_{u_kv_j}G_2$ then, $irr_t(G') = irr_t(G'')$.
\end{corollary}
\begin{proof}
Begin the proof by choosing any vertex degree value $t_1$ in the degree sequence of $G_1$ and identify largest vertex index say, $i$ for which $d_{G_1}(u_i) = t_1$. Similarly, choose any vertex degree value $t_2$ in the degree sequence of $G_2$ and identify largest vertex index say, $l$ for which $d_{G_2}(v_l) = t_2$. Here, we have to consider the following cases.

\ni{\em Case 1:} With respect to $G' = G_1\rightsquigarrow_{u_iv_l}G_2$, using definition \ref{D-VSS1}, set the values as follows.

\begin{enumerate}\itemsep0mm
\item[(i)] $|V_1| = a = i-1$,
\item[(ii)] $|V_2| = b = n-i$,
\item[(iii)] $|V_3| = a^{\ast} = j$,
\item[(iv)] $|V_4| = b^{\ast} = m-j$
\item[(v)] $|V_5| = c = l - 1$,
\item[(vi)] $|V_6| = d = m- l$,
\item[(vii)] $|v_7| = c^{\ast} = k$,
\item[(viii)] $|V_8| = d^{\ast} = n-k$.
\end{enumerate}

Therefore, we have $2(n+m) -2((n-i) + (m-j) +(m- l) + (n-k)) -2 = 2(i+j+k+l -(n+m)) - 2$.

\ni{\em Case 2:} In respect of $G'' = G_1\rightsquigarrow_{u_kv_j}G_2$ and using definition \ref{D-VSS1}, set the values as follows.

\begin{enumerate}\itemsep0mm
\item[(i)] $|V_1| = a = k-1$,
\item[(ii)] $|V_2| = b = n- k$,
\item[(iii)] $|V_3| = a^{\ast} = l$,
\item[(iv)] $|V_4| = b^{\ast} = m-l$
\item[(v)] $|V_5| = c = j - 1$,
\item[(vi)] $|V_6| = d = m- j$,
\item[(vii)] $|v_7| = c^{\ast} = i$,
\item[(viii)] $|V_8| = d^{\ast} = n- i$.
\end{enumerate}

Therefore, here we have $2(n+m)-2((n-k)+(m-l)+(m-j)+(n-i))-2=2(k+l+j+i-(n+m))-2$.

Since Case 1 and Case 2 yields the same result, the result $irr_t(G') = irr_t(G'')$ follows from \ref{Thm-2.1}.
\end{proof}

An immediate consequence of Corollary \ref{Cor-2.6} is that for regular graphs $G_1$ and $G_2$ we have $irr_t(G_1\rightsquigarrow_{vu}G_2)_{\substack{u \in V(G_1)\\ v \in V(G_2)}}$ is a constant. This result is proved in the following proposition.

\begin{proposition}
For the regular graphs $G_1, G_2$ on $n, m$ vertices respectively with $d_{G_1}(u) \geq d_{G_2}(v)$ we have
\begin{equation*} 
irr_t(G_1\rightsquigarrow_{uv}G_2)=
\begin{cases}
2(n+m) - 2, & \text {if $d_{G_1}(u) = d_{G_2}(v)$,}\\ 
n\cdot m|d_{G_1}(u)- d_{G_2}(v)| + 2(n-1), & \text {if $d_{G_1}(u) > d_{G_2}(v)$.}
\end{cases}
\end{equation*}
\end{proposition}
\begin{proof}
The proof follows immediately from definitions \ref{D-VSS1}, Remark \ref{Rem-2.3} and Corollary \ref{Cor-2.6}.
\end{proof}

We note that if $G_1$ and $G_2$ are of equal $k$-regularity, then $irr_t(G_1\rightsquigarrow_{uv}G_2)$ is independent of the \textit{$k$- degree} of the vertices.

\section{Total Irregularity Due to Edge-Transformation}

Consider a graph $G$ on $n = l_1 + l_2$ vertices and a cut edge $u_1v_1$. Let $G = (G_1 \cup G_2) + u_1v_1$, $u_1 \in V(G_1) = \{u_i: 1 \leq i \leq l_1\}$ and $v_1 \in V(G_2) = \{v_i: 1 \leq i \leq l_2\}$. Edge-transformation with respect to $u_1$ will be the graph $G^{u_iv_1}$ obtained by deleting the edge $u_1v_1$ and adding the edge $u_iv_1$ for any $i \neq 1$. We call $G_1$ the {\em master graph} and $G_2$ the {\em slave graph}.

\vspace{0.2cm}

Let us now introduce the notion of edge-transformation partitioning of a vertex set of a given graph as follows. 

\begin{definition}\label{D-VSS2a}{\rm 
The \textit{edge-transformation partitioning} of the vertex set $V(G)$ of a graph $G$ on $n$ vertices with at least one cut edge say $u_1v_1$, is defined to be $V_h = \{u_i, v_k:d_{G_1}(u_i) = d_{G_1}(u_1) -1\}$ and $d_{G_2}(v_k) = d_{G_1}(u_1) -1\}\cup \{u_1\}, h= |V_h|$, and $V_s = \{ u_i, v_k: d_{G_1}(u_i) > d_{G_1}(u_1)-1\}$ and $d_{G_2}(v_k) > d_{G_1}(u_1)-1\}, s=|V_s|$ and $V_t = \{ u_i, v_k: d_{G_1}(u_i) < d_{G_1}(u_1)-1\}$ and $d_{G_2}(v_k) < d_{G_1}(u_1)-1\}, t=|V_t|$.}
\end{definition}

\ni Invoking Definition \ref{D-VSS2a}, we now define certain vertex sets in $G$ as given below.

\begin{definition}{\rm 
We define the certain important sets as follows. Let $V_{s_1} = \{u_j, v_k: d_{G_1}(u_j) \leq d_{G_1}(u_i)$ and $d_{G_2}(v_k) \leq d_{G_1}(u_i)$ and $u_j, v_k \in V_s \}, m = |V_{s_1}|$, and $V_{s_2} = \{u_j, v_k: d_{G_1}(u_j) > d_{G_1}(u_i)$ and $d_{G_2}(v_k) > d_{G_1}(u_i)$ and $u_j, v_k \in V_s \}, l = |V_{s_2}|$, and $V_{t_1} = \{u_j, v_k: d_{G_1}(u_j) \leq d(u_i)$ and $d_{G_2}(v_k) \leq d_{G_1}(u_i)$ and $u_j, v_k \in V_t \}, m_1 = |V_{t_1}|$, and $V_{t_2} = \{u_j, v_k: d_{G_1}(u_j) > d_{G_1}(u_i)$ and $d_{G_2}(v_k) > d_{G_1}(u_i)$ and $u_j, v_k \in V_t \}, l_1 = |V_{t_2}|$.}
\end{definition}

\ni In view of the above definitions, we propose the following theorem

\begin{theorem}\label{Thm-3.3}
For a graph $G$ with a cut edge $u_1v_1$, let $G-u_1v_1 = G_1 \cup G_2$. After edge-transformation in respect of $v_1$ we have
\begin{equation*} 
irr_t(G^{u_iv_1})=
\begin{cases}
irr_t(G), & \text{if} \quad d_{G_1}(u_i) = d_{G_1}(u_1) -1,\\ 
irr_t(G) + 2m, & \text{if}\quad d_{G_1}(u_i) > d_{G_1}(u_1)-1,\\
irr_t(G) - 2(h + l_1), & \text{if}\quad d_{G_1}(u_i) < d_{G_1}(u_1)-1.
\end{cases}
\end{equation*} 
\end{theorem}
\begin{proof}
If $d_{G_1}(u_i) = d_{G_1}(u_1) -1$,  then reducing $d_{G_1}(u_1)$ by 1, reduces the partial sum
$\sum \limits_{\substack{j=1\\w_j \in V_h}}^{h}|d_{G_1}(u_1)-d(w_j)|$ by exactly $(h-1)$. It also increases the partial sum $\sum \limits_{\substack{j=1\\w_j \in V_s}}^{s}|d_{G_1}(u_1)-d(w_j)|$ by exactly $s$ and finally it reduces the the partial sum $\sum \limits_{\substack{j=1\\w_j \in V_t}}^{t}|d_{G_1}(u_1)-d(w_j)|$ by exactly $t$. 

\ni {\em Case 1:} By increasing $d_{G_1}(u_i), u_i \in V_h$ by $1$, the partial sum $\sum \limits_{\substack{j=1\\w_j \in V_h}}^{h}|d_{G_1}(u_i)-d(w_j)|$ increases by exactly $(h-1)$. It also decreases the partial sum $\sum \limits_{\substack{j=1\\w_j \in V_s}}^{s}|d_{G_1}(u_i)-d(w_j)|$ by exactly $s$ and finally it increases the the partial sum $\sum \limits_{\substack{j=1\\w_j \in V_t}}^{t}|d_{G_1}(u_i)-d(w_j)|$ by exactly $t$. Hence, the result, $irr_t(G^{u_iv_1})=irr_t(G)-(h -1)+s-t +((h-1)-s+t)=irr_t(G)$ follows.

\ni {\em Case 2:} By increasing $d_{G_1}(u_i), u_i \in V_s$ by 1, the partial sum $\sum \limits_{\substack{j=1\\w_j \in V_h}}^{h}|d_{G_1}(u_i)-d(w_j)|$ increases by exactly $h$. It also changes the partial sum $\sum \limits_{\substack{j=1\\w_j \in V_s}}^{s}|d_{G_1}(u_i) - d(w_j)|$ by exactly $(m-1)-l$ and finally it increases the the partial sum $\sum \limits_{\substack{j=1\\w_j \in V_t}}^{t}|d_{G_1}(u_i)-d(w_j)|$ by exactly $t$. Hence, the result, $irr_t(G^{u_iv_1}) = irr_t(G) - (h -1) + s - t + h + (m-1) - l + t =  irr_t(G) + 2m$ follows.

\ni {\em Case 3:} By increasing $d_{G_1}(u_i), v_i \in V_t$ by 1, the partial sum $\sum \limits_{\substack{j=1\\w_j \in V_t}}^{h}|d_{G_1}(u_i)-d(w_j)|$ decreases by exactly $h$. It also decreases the partial sum $\sum \limits_{\substack{j=1\\w_j \in V_t}}^{s}|d_{G_1}(u_i)-d(w_j)|$ by exactly $s$ and finally it changes the the partial sum $\sum \limits_{\substack{j=1\\w_j \in V_t}}^{t}|d_{G_1}(u_i)-d(w_j)|$ by exactly $(m_1-1)-l_1$.

Hence, the result $irr_t(G^{u_iv_1})=irr_t(G)-(h-1)+s-t-(h-1)-2-s+(m_1-1)-l_1=irr_t(G)-2(h+l_1)$ follows.
\end{proof}

It is to be noted Theorem \ref{Thm-3.3} provides an alternate proof for the following that lemma provided in \cite{ZYY}.

\begin{lemma}
{\rm \cite{ZYY}} Let $G'$ be the graph obtained from $G$ by branch-transformation from $u$ to $v$. Then $irr_t(G) > irr_t(G')$.
\end{lemma}

\ni Theorem \ref{Thm-3.3} can be extended to multi graphs also as explained in the following result.

\begin {corollary}
If multiple edges or loops are allowed in the graph or if edge-transformation is performed in a simple graph without a cut edge to give $G^{w_iv_1}_{w_i \in G}$ then, we have 
\begin{equation*} 
irr_t(G^{u_iv_1})=
\begin{cases}
irr_t(G), & \text{if} \quad d_{G_1}(u_i) = d_{G_1}(u_1) -1,\\ 
irr_t(G) + 2m, & \text{if}\quad d_{G_1}(u_i) > d_{G_1}(u_1)-1,\\
irr_t(G) - 2(h + l_1), & \text{if}\quad d_{G_1}(u_i) < d_{G_1}(u_1)-1.
\end{cases}
\end{equation*} 
\end{corollary}
\begin{proof}
The proof of this theorem follows immediately as a consequence of Theorem \ref{Thm-3.3}.
\end{proof}

\section{Total Irregularities of Directed Graphs}\label{Sec-4}

In this section, we extend the concept of total irregularities of graphs mentioned in above sections to directed graphs. Since the edges of a digraph $D$ are directed edges and the vertices of $D$ has two types of degrees, in-degrees and out-degrees, we need to define two types of total irregularities for a digraph, which are called total in-degree irregularities and total out-degree irregularities.

\ni Let the vertices of a simple directed graph $D^\rightarrow$ on $n$ vertices be labelled as $v_i; i= 1,2,3,\ldots,n$ and let $d^+_{D^\rightarrow}(v_i)=d^+(v_i)$ and $d^-_{D^\rightarrow}(v_i)=d^-(v_i)$. Then, the notion of total in-irregularity of a given directed graph is introduced as follows.

\begin{definition}{\rm 
The {\em total in-irregularity} of a directed graph $D$ with respect to the in-degree of all vertices of $D$, denoted by $irr_t^-(D^\rightarrow)$, is defined as $irr_t^-(D^\rightarrow)=\frac{1}{2}\sum\limits_{i=1}^{n}\sum\limits_{j=1}^{n}|d^-(v_i)-d^-(v_j)| =\sum\limits_{i=1}^{n} \sum \limits_{j=i+1}^{n}|d^-(v_i) - d^-(v_j)|$ or $\sum \limits_{i=1}^{n-1} \sum \limits_{j=i+1}^{n}|d^-(v_i) - d^-(v_j)|$.}
\end{definition}

\ni Similarly, the total out-irregularity of a digraph can also be defined as follows.

\begin{definition}{\rm 
The {\em total out-irregularity} of a directed graph $D$ with respect to the out-degree of all vertices of $D$, denoted by $irr_t^+(D^\rightarrow)$, is defined as $irr_t^+(G^\rightarrow)=\frac{1}{2}\sum\limits_{i=1}^{n} \sum \limits_{j=1}^{n}|d^+(v_i) - d^+(v_j)| = \sum \limits_{i=1}^{n} \sum \limits_{j=i+1}^{n}|d^+(v_i)-d^+(v_j)|$ or $\sum \limits_{i=1}^{n-1} \sum \limits_{j=i+1}^{n}|d^+(v_i)-d^+(v_j)|$. }
\end{definition}

Re-orientation of an arc or arc-transformation of an arc will find application in most classical applications of directed graphs like tournaments, transportation problems, flow analysis or alike.
 
\subsection{Total Irregularities of Directed Paths and Cycles}

The total in-irregularity and the total out-irregularity of a directed path are determined in the following proposition.

\begin{proposition}
For a directed path $P_n^\rightarrow$ which is consecutively directed from left to right for which vertices $v_1, v_n$ are called the \textit{start-vertex} and the \textit{end-vertex} respectively, we have 
\begin{enumerate}\itemsep0mm
\item[(i)] $irr_t^-(P_n^\rightarrow) = irr_t^+(P_n^\rightarrow)=n-1$,
\item[(ii)] 
\begin{equation*} 
irr^-_t(P_n^{\rightarrow}) =
\begin{cases}
n-1, & \text{if the orientation of}~ (v_1, v_2) \text{is reversed}, \\
3n-5, & \text{if the orientation of}~ (v_i, v_{i+1}), 2 \leq i \leq (n-1) \\ & \text{is reversed}
\end{cases}
\end{equation*}
\item[(iii)] \begin{equation*} 
irr^+_t(P_n^{\rightarrow}) =
\begin{cases}
n-1 , & \text {if the orientation of}~ (v_{n-1}, v_n) \text{is reversed},\\  
3n-5, & \text {if the orientation of}~ (v_i,v_{i+1}), \leq i\leq n-2 \\ &   \text{is reversed}.
\end{cases}
\end{equation*}
\end{enumerate}
\end{proposition}
\begin{proof}
The proof is obvious from the definition of total in-irregularity and total out-irregularity of a given digraph.
\end{proof}

The total in-irregularity and the total out-irregularity of a directed cycle are determined in the following proposition.

\begin{proposition}
For a directed cycle $C_n^\rightarrow$ which is consecutively directed clockwise we have 
\begin{enumerate}\itemsep0mm
\item[(i)] $irr_t^-(C_n^\rightarrow)=irr_t^+(C_n^\rightarrow)=0$,
\item[(ii)] $irr_t^-(C_n^\rightarrow)=irr_t^+(C_n^\rightarrow)=2(n-1)$, if we reverse the orientation of any arc.
\end{enumerate}
\end{proposition}
\begin{proof}
The proof is obvious from the definition of total in-irregularity and total out-irregularity of a given digraph.
\end{proof}

\ni Through a simple change of Definition \ref{D-VSS2a} the in-arc-transformation partitioning in respect of $v_1$ and the out-arc-transformation partitioning in respect of $v_1$ can be defined.

\begin{definition}\label{Def-4.3}{\rm 
The \textit{in-arc-transformation partitioning} with respect to a vertex $v_i$ of the vertex set $V(G)$ of a simple connected directed graph $G^\rightarrow$ on $n$ vertices is defined to be $V_h = \{v_i:d^-(v_i) = (d^-(v_1) -1)\} \cup \{v_1\}, h= |V_h|$, and $V_s = \{ v_i: d^-(v_i) > (d^-(v_1)-1)\}, s=|V_s|$ and $V_t = \{ v_i: d^-(v_i) < (d^-(v_1)-1)\}, t=|V_t|$.}
\end{definition}

\ni In view of Definition \ref{Def-4.3}, we define the following sets

\begin{definition}\label{Def-4.4}{\rm 
Invoking the above definition, some vertex sets of a given digraph are defined as follows. $V_{s_1} = \{v_j:| d^-(v_j) \leq d^-(v_i), v_j \in V_s\}, m = |V_{s_1}|$ and $V_{s_2} = \{v_j: d^-(v_j) > d^-(v_i), v_j \in V_s\}, l = |V_{s_2}|$ and $V_{t_1} = \{v_j: d^-(v_j) \leq d^-(v_i), v_j \in V_t \}, m_1 = |V_{t_1}|$ and $V_{t_2} = \{v_j: d^-(v_j) > d^-(v_i), v_j \in V_t\}, l_1 = |V_{t_2}|$.}
\end{definition}

\begin{definition}\label{Def-4.5} {\rm 
The \textit{out-arc-transformation partitioning} with respect to a vertex $v_i$ of the vertex set $V(G^\rightarrow)$ of a simple connected directed graph $G^\rightarrow$ on n vertices is defined to be $V_{h^{\ast}} = \{v_i:d^+(v_i) = (d^+(v_1) -1)\} \cup \{v_1\}, h^{\ast}= |V_{h^{\ast}}|$, and $V_{s^{\ast}} = \{ v_i: d^+(v_i) > (d^+(v_1)-1)\}, s^{\ast}=|V_{s^{\ast}}|$ and $V_{t^{\ast}} = \{ v_i: d^+(v_i) < (d^+(v_1)-1)\}, t^{\ast}=|V_{t^{\ast}}|$.}
\end{definition}

\ni In view of Definition \ref{Def-4.5}, we define the following sets

\begin{definition}\label{Def-4.6}{\rm 
Invoking the above definition, some vertex sets of a given digraph are defined as follows. $V_{s^{\ast}_1} = \{v_j: d^+(v_j) \leq d^+(v_i), v_j \in V_s\}, m^{\ast} = |V_{s^{\ast}_1}|$ and $V_{s^{\ast}_2} = \{v_j: d^+(v_j) > d^+(v_i), v_j \in V_s\}, l^{\ast} = |V_{s^{\ast}_2}|$ and $V_{t^{\ast}_1} = \{v_j: d^+(v_j) \leq d^+(v_i), v_j \in V_t \}, m^{\ast}_1 = |V_{t^{\ast}_1}|$ and $V_{t^{\ast}_2} = \{v_j: d^+(v_j) > d^+(v_i), v_j \in V_t\}, l^{\ast}_1 = |V_{t^{\ast}_2}|$. }
\end{definition}

\ni Analogous to Theorem \ref{Thm-3.3}, we propose the following result.

\begin{proposition}
Consider a simple connected directed graph $G$. After in-arc-transformation in respect of $v_1$ we have 
\begin{enumerate}
\item[(i)]
\begin{equation*} 
irr_t^-(G^{v_iu_1})=
\begin{cases}
irr_t^-(G) , & \text{if}\quad d^-(v_i) = d^-(v_1) -1\\ 
irr_t^-(G) + 2m, & \text{if}\quad d^-(v_i) > d^-(v_1)-1, \\
irr_t^-(G) - 2(h + l_1), & \text{if}\quad d^-(v_i) < d^-(v_1)-1
\end{cases}
\end{equation*} 
and 
\item[(ii)]
\begin{equation*} 
irr_t^+(G^{v_iu_1})=
\begin{cases}
irr_t^+(G) , & \text {if}\quad d^+(v_i)=d^+(v_1)-1,\\
irr_t^+(G) + 2m^{\ast}, & \text {if}\quad d^+(v_i) > d^+(v_1)-1,\\
irr_t^+(G) - 2(h^{\ast} + l^{\ast}_1), & \text {if}\quad d^+(v_i) < d^+(v_1)-1.
\end{cases}
\end{equation*} 
\end{enumerate}
\end{proposition}
\begin{proof}
The proof is similar to Theorem \ref{Thm-3.3}.
\end{proof}

\subsection{Total Irregularities of Directed Complete Graphs}

In this section, we initiate a study on the two types of irregularities of directed complete graphs. Consider a complete undirected graph $K_n$ and label the vertices $v_1, v_2, v_3, \ldots, v_n$. Assign direction the edges of $K_n$ to get a directed graph, with $K_n$ as its underlying graph, in such a way that the edge $v_iv_j$ becomes the  arc $(v_i,v_j)$ of this directed graph if $i<j$. We denote this directed graph by $K^\rightarrow_n$. The following lemma discusses the two types of irregularities of $K^\rightarrow_n$.  

\begin{lemma}\label{Lem-4.8}
 For the directed complete graph $K^{\rightarrow}_n$, the total irregularities are given by $irr_t^+ (K^{\rightarrow}_n) = irr_t^- (K^\rightarrow_n)= \sum \limits_{i=1}^{n-1}\sum \limits_{j=1}^{i}j =\frac{1}{6}n(n^2-1)$.
\end{lemma}
\begin{proof}
The orientation results in an in-degree sequence $(0,1,2,\ldots,(n-1))$ and an out-degree sequence $(n-1,n-2,n-3,\ldots,0)$. Choose the $k$-th entry of the in-degree sequence. We know that the $k$-th term is given by $\sum\limits_{j=k+1}^{n} |d^-(v_k)-d^-(v_j)|=\sum \limits_{i=1}^{n-(k+1)}i$. Also, we have $irr_t^-=\sum \limits_{i=1}^{n-1} \sum \limits_{j=i+1}^{n}|d^-(v_i) - d^-(v_j)|$ and hence $irr_t^-(K^\rightarrow_n)=\sum\limits_{i=1}^{n-1}i+\sum \limits_{i=1}^{n-2}i+\ldots +\sum\limits_{i=1}^{n-(n-1)}i=\sum\limits_{i=1}^{n-1}\sum\limits_{j=1}^{i}j=\frac{1}{6}n(n^2-1)$. Furthermore, since the out-degree sequence is a mirror image of the in-degree sequence and $irr_t^+ = \sum \limits_{i=1}^{n-1} \sum \limits_{j=i+1}^{n}|d^+(v_i) - d^+(v_j)|$, the result follows similarly.
\end{proof}

\ni A general application this study can be the following. 

Consider any connected undirected graph $G$ on $n$ vertices and label its vertices randomly by $v_1, v_2, v_3, \ldots, v_n$. Assign direction to the edges of the graph $G$ to be arcs according to the condition mentioned above and refer to the directed graph as the {\em root directed graph}, $G^\rightarrow_{root-graph}$. Then, calculate both $irr_t^+(G^\rightarrow_{root-graph})$ and $irr_t^-(G^\rightarrow_{root-graph})$. In a {\em derivative graph} $G^\rightarrow_{derivative}$ identify all arcs which were re-oriented or subjected to arc-transformation and apply the applicable results to recursively determine the total in-irregularity and total out-irregularity.

Consider the complete bipartite graph $K_{(m, n)}$ and call the $m$ vertices in the first bipartition by {\em left-side vertices} and the $n$ vertices in the second bipartition by {\em right-side vertices}. Assign directions to the edges of $K_{m,n}$ strictly from {\em left-side vertices} to \textit{right-side vertices} to obtain $K^{l \rightarrow r}_{m, n}$.

\begin{proposition}\label{Prop-4.9}
For the directed graph $K^{l \rightarrow r}_{m,n}$, we have $irr^-_t(K^{l \rightarrow r}_{m,n}) = m^2n$ and $irr^+_t(K^{l \rightarrow r}_{m,n}) = mn^2$.
\end{proposition}
\begin{proof}
The orientation of the directed complete bipartite graph $K^{l \rightarrow r}_{m,n}$ results in the in-degree sequence $(\underbrace{0, 0, \ldots, 0,}_{m-entries} \underbrace{m, m, \ldots, m}_{n-enties})$ and the out-degree sequence $( \underbrace{n,n,\ldots,n,}_{m-entries}\\ \underbrace{0, 0, \ldots, 0}_{n-enties})$. Here, we have the following cases.

\ni \textit{Case 1:} For the above mentioned in-degree sequence of $K^{l \rightarrow r}_{m,n}$, we have the sum $\sum\limits_{i=1}^{(m+n)-1}\sum\limits_{j=(i+1)}^{(m+n)}|d^-(v_i)-d^-(v_j)|$ results in the value $m$, ($mn$ times) and $0$, ($(m+n)-2$ times). Hence, $irr_t^-(K^{l \rightarrow r}_{(m,n)}) = m^2n$.

\ni \textit{Case 2:} For the above mentioned out-degree sequence of $K^{l \rightarrow r}_{m,n}$, we have the sum $\sum\limits_ {i=1}^{(m+n)-1}\sum\limits_{j= (i+1)}^{(m+n)}|d^+(v_i)-d^+(v_j)|$ results in the value $n$, ($mn$ times) and $0$, ($(m+n)-2$) times).  Hence, $irr_t^+(K^{l \rightarrow r}_{(1, n)}) = mn^2$. This completes the proof.
\end{proof}

Invoking from Proposition \ref{Prop-4.9}, we note that for the directed bipartite graph $K^{l \rightarrow r}_{1,n}$, we have $irr^-_t(K^{l \rightarrow r}_{1,n}) = n$ and $irr^+_t(K^{l \rightarrow r}_{(1, n)}) = n^2$ and $irr^-_t(K^{l \rightarrow r}_{m,1}) = m^2$ and $irr^+_t(K^{l\rightarrow r}_{m,1})=m$.

\ni The following is a challenging and interesting problem in this context.

\begin{problem}{\rm 
Describe an efficient algorithm to determine $irr_t^- (G^\rightarrow_{derivative})$ and $irr_t^+(G^\rightarrow_{derivative})$  from $irr_t^-(G^\rightarrow_{root-graph})$ and $irr_t^+(G^\rightarrow_{root-graph})$.}
\end{problem}

\section{Conclusion}

In this paper, we have studied certain types of total irregularities of certain graphs and digraphs. More problems in this area still remain unsettled. More studies on different types of irregularities for different graph classes, graph operations, graph products and on certain associated graphs such as line graphs and total graphs of given graphs and digraphs remain open. All these facts indicates that there is a wide scope for further investigations in this area.

\end{document}